\documentclass[12pt,a4paper]{article}

\usepackage[a4paper,margin=27mm]{geometry}
\usepackage[T1]{fontenc}
\usepackage[utf8]{inputenc}
\usepackage{lmodern}
\usepackage{orcidlink}
\usepackage{microtype}
\usepackage{amsmath,amssymb,amsthm,mathtools}
\usepackage{enumitem}
\usepackage{booktabs}
\usepackage{xcolor}
\usepackage{graphicx}
\usepackage{tikz}
\usepackage{pgfplots}
\pgfplotsset{compat=1.18}
\usepackage{float}
\usetikzlibrary{arrows.meta,calc,positioning,patterns,shapes.geometric,3d}
\usepackage{hyperref}
\hypersetup{colorlinks=true,linkcolor=blue!50!black,citecolor=blue!50!black,urlcolor=blue!50!black}

\newcommand{\Rtwo}{R_2}
\newcommand{\DeltaTwo}{\Delta_2}
\newcommand{\dd}{\,\mathrm{d}}
\newcommand{\argmin}{\operatorname*{arg\,min}}
\newcommand{\argmax}{\operatorname*{arg\,max}}

\newcommand{\RR}{\mathbb{R}}
\newcommand{\vol}{\operatorname{vol}}
\newcommand{\boxop}{\operatorname{box}}
\newcommand{\eps}{\varepsilon}
\newcommand{\muvarphi}{\mu_\varphi}

\theoremstyle{plain}
\newtheorem{theorem}{Theorem}[section]
\newtheorem{lemma}[theorem]{Lemma}

\newtheorem{corollary}[theorem]{Corollary}

\theoremstyle{definition}
\newtheorem{definition}[theorem]{Definition}
\newtheorem{problem}[theorem]{Problem}

\theoremstyle{remark}
\newtheorem{remark}[theorem]{Remark}

\newenvironment{pseudocode}[1]{%
  \begin{figure}[H]
  \centering
  \begin{minipage}{0.92\textwidth}\small
  \hrule\medskip
  \textbf{#1}\par\medskip
}{%
  \medskip\hrule
  \end{minipage}
  \end{figure}
}

\title{Three-Objective Integral R2 Subset Selection: NP-Hardness and Submodular Approximation}
\author{Michael T. M. Emmerich\orcidlink{0000-0002-7342-2090}}
\date{\today}

\begin{document}
\maketitle

\begin{abstract}
Selecting a fixed number of representative points from a finite Pareto-front approximation is a fundamental post-processing task in multiobjective optimization. This paper studies this problem for the integral $\Rtwo$ indicator in three objectives, where the indicator is defined as the integral of the lower envelope of weighted Tchebycheff scalarizations over the two-dimensional weight simplex. We provide two complementary algorithmic results. On the positive side, we show that the integral $\Rtwo$ improvement with respect to any fixed baseline is a monotone submodular set function. For the usual ideal-point based $\Rtwo$ indicator, with the ideal point fixed, this yields a direct gap-reduction guarantee: greedy selection closes at least a $(1-1/e)$-fraction of the maximum possible $\Rtwo$ gap between a fixed dominated anchor value and the best cardinality-$k$ value. We also give a tested greedy implementation that evaluates exact integral $\Rtwo$ values by subdivision, with worst-case running time $O(n^6)$. On the negative side, we prove that exact fixed-cardinality subset selection is NP-hard already in three objectives. The hardness proof uses a perspective transformation that maps Tchebycheff-shadow improvements to a weighted anchored-box union problem with density $(x_1+x_2+x_3)^{-4}$, and then adapts the three-dimensional anchored-box construction of Bringmann, Cabello, and Emmerich. Together, these results separate the tractable two-objective case from the three-objective case while identifying a principled approximation route based on submodular optimization.
\end{abstract}


\section{Introduction}

Finite Pareto-front approximations are often produced as an intermediate product of evolutionary multiobjective optimization, Bayesian multiobjective optimization, interactive decision support, and multicriteria design workflows.  In many applications the archive is larger than the number of alternatives that can be inspected by a decision maker or used in a subsequent expensive simulation, certification, or robustness analysis.  This leads to a basic algorithmic post-processing problem: given a finite approximation $P$ and a cardinality $k$, select a subset of $k$ points that preserves as much indicator quality as possible.

We study this problem for the integral $\Rtwo$ indicator in three objectives.  The indicator is based on weighted Tchebycheff scalarization and can be interpreted as an average, over a continuum of preference weights, of the best scalarized loss achieved by the selected set.  We use a minimization convention and assume that the utopian point has been translated to the origin.  For a point $p=(p_1,p_2,p_3)\in\RR^3_{>0}$ and a weight vector $w$ in the two-dimensional simplex
\[
\DeltaTwo=\{w\in\RR^3_{\geq 0}:w_1+w_2+w_3=1\},
\]
the weighted Tchebycheff scalarization is
\[
 g_p(w)=\max\{w_1p_1,w_2p_2,w_3p_3\}.
\]
For a finite set $S$ of objective vectors, the Tchebycheff envelope is
\[
 \tau_S(w)=\min_{p\in S}g_p(w),
\]
and the integral $\Rtwo$ value is
\[
 \Rtwo(S)=\int_{\DeltaTwo}\tau_S(w)\dd w.
\]
Thus the set is judged by the lower envelope of all pointwise Tchebycheff losses.  The integration is not over a finite weight grid: every normalized weight contributes.  In a two-objective problem this means integrating an envelope over a line segment of weights.  In a three-objective problem it means integrating an envelope over the triangular simplex.  Smaller values are better.

The geometric picture used in this report is the Tchebycheff-shadow representation.  In two objectives, each archive point defines a V-shaped loss curve over the weight interval, and the integral $\Rtwo$ value is the area under the lower envelope of these curves.  In three objectives, each archive point defines a piecewise-linear surface over the weight simplex, and the integral $\Rtwo$ value is the volume under the lower envelope of these surfaces.  Adding a point lowers the old envelope on the weights for which the new point becomes best; the vertical volume between the old and new envelopes is the point's Tchebycheff-shadow improvement.

\begin{figure}[H]
\centering
\begin{minipage}{0.48\textwidth}
\centering
\begin{tikzpicture}
\begin{scope}[xshift=0cm]
\begin{axis}[
  width=0.98\linewidth, height=0.68\linewidth, axis lines=left,
  xmin=0, xmax=0.72, ymin=0, ymax=0.68,
  xlabel={$f_1$}, ylabel={$f_2$}, grid=both,
  grid style={line width=.1pt, draw=gray!18}, major grid style={line width=.2pt, draw=gray!32},
  tick label style={font=\scriptsize}, label style={font=\small}, clip=false]
  \addplot[gray!70, thick, densely dashed] coordinates {(0.10,0.60) (0.18,0.43) (0.30,0.30) (0.45,0.18) (0.62,0.10)};
  \draw[violet!65, thick] (axis cs:0,0.60) -- (axis cs:0.10,0.60) -- (axis cs:0.10,0);
  \draw[blue!60!black, thick]   (axis cs:0,0.43) -- (axis cs:0.18,0.43) -- (axis cs:0.18,0);
  \draw[orange!85!black, thick] (axis cs:0,0.30) -- (axis cs:0.30,0.30) -- (axis cs:0.30,0);
  \draw[green!55!black, thick] (axis cs:0,0.18) -- (axis cs:0.45,0.18) -- (axis cs:0.45,0);
  \draw[gray!70!black, thick] (axis cs:0,0.10) -- (axis cs:0.62,0.10) -- (axis cs:0.62,0);
  \addplot[only marks, mark=*, mark size=2.1pt, violet!80!black] coordinates {(0.10,0.60)};
  \addplot[only marks, mark=*, mark size=2.1pt, blue!70!black] coordinates {(0.18,0.43)};
  \addplot[only marks, mark=*, mark size=2.1pt, orange!90!black] coordinates {(0.30,0.30)};
  \addplot[only marks, mark=*, mark size=2.1pt, green!60!black] coordinates {(0.45,0.18)};
  \addplot[only marks, mark=*, mark size=2.1pt, gray!80!black] coordinates {(0.62,0.10)};
  \node[font=\scriptsize, fill=white, inner sep=0.8pt, anchor=south west, text=violet!80!black] at (axis cs:0.105,0.602) {$p_1$};
  \node[font=\scriptsize, fill=white, inner sep=0.8pt, anchor=south west, text=blue!70!black] at (axis cs:0.185,0.432) {$p_2$};
  \node[font=\scriptsize, fill=white, inner sep=0.8pt, anchor=south west, text=orange!90!black] at (axis cs:0.305,0.302) {$p_3$};
  \node[font=\scriptsize, fill=white, inner sep=0.8pt, anchor=south west, text=green!60!black] at (axis cs:0.455,0.182) {$p_4$};
  \node[font=\scriptsize, fill=white, inner sep=0.8pt, anchor=south west, text=gray!80!black] at (axis cs:0.625,0.102) {$p_5$};
  \node[font=\scriptsize, anchor=north east] at (axis cs:0,0) {$z^+$};
\end{axis}
\end{scope}
\end{tikzpicture}
\par\smallskip
{\small\textbf{(a)} Biobjective Pareto-front approximation\par}
\end{minipage}\hfill
\begin{minipage}{0.48\textwidth}
\centering
\begin{tikzpicture}
\begin{scope}[xshift=0cm]
\begin{axis}[
  width=0.98\linewidth, height=0.68\linewidth, axis lines=left,
  xmin=0, xmax=1, ymin=0, ymax=0.32,
  xlabel={$\lambda$ in $(\lambda,1-\lambda)$}, ylabel={Tchebycheff loss},
  xtick={0,0.2,0.4,0.6,0.8,1.0}, ytick={0,0.08,0.16,0.24,0.32},
  grid=both, grid style={line width=.1pt, draw=gray!18}, major grid style={line width=.2pt, draw=gray!32},
  tick label style={font=\scriptsize}, label style={font=\small}, clip=true]
  \addplot[draw=none, fill=blue!18, domain=0:1, samples=300]
    {min(min(min(min(max(0.10*x,0.60*(1-x)),max(0.18*x,0.43*(1-x))),max(0.30*x,0.30*(1-x))),max(0.45*x,0.18*(1-x))),max(0.62*x,0.10*(1-x)))} \closedcycle;
  \addplot[violet!80!black, densely dashed, domain=0:1, samples=100] {max(0.10*x,0.60*(1-x))};
  \addplot[blue!70!black, densely dashed, domain=0:1, samples=100] {max(0.18*x,0.43*(1-x))};
  \addplot[orange!90!black, densely dashed, domain=0:1, samples=100] {max(0.30*x,0.30*(1-x))};
  \addplot[green!60!black, densely dashed, domain=0:1, samples=100] {max(0.45*x,0.18*(1-x))};
  \addplot[gray!80!black, densely dashed, domain=0:1, samples=100] {max(0.62*x,0.10*(1-x))};
  \node[font=\scriptsize, fill=white, inner sep=0.8pt, text=violet!80!black] at (axis cs:0.52,0.288) {$\tau_1$};
  \node[font=\scriptsize, fill=white, inner sep=0.8pt, text=blue!70!black] at (axis cs:0.29,0.305) {$\tau_2$};
  \node[font=\scriptsize, fill=white, inner sep=0.8pt, text=orange!90!black] at (axis cs:0.88,0.270) {$\tau_3$};
  \node[font=\scriptsize, fill=white, inner sep=0.8pt, text=green!60!black] at (axis cs:0.86,0.235) {$\tau_4$};
  \node[font=\scriptsize, fill=white, inner sep=0.8pt, text=gray!80!black] at (axis cs:0.73,0.245) {$\tau_5$};
  \addplot[blue!70!black, very thick, domain=0:1, samples=300]
    {min(min(min(min(max(0.10*x,0.60*(1-x)),max(0.18*x,0.43*(1-x))),max(0.30*x,0.30*(1-x))),max(0.45*x,0.18*(1-x))),max(0.62*x,0.10*(1-x)))};
  \node[font=\scriptsize, fill=white, inner sep=1.0pt, anchor=west] at (axis cs:0.28,0.040) {area = integral $R_2$};
\end{axis}
\end{scope}
\end{tikzpicture}
\par\smallskip
{\small\textbf{(b)} Biobjective Tchebycheff shadows\par}
\end{minipage}

\par\medskip

\begin{minipage}{0.485\textwidth}
\centering
\begin{tikzpicture}[x={(1cm,0cm)},y={(0.48cm,0.19cm)},z={(0cm,0.95cm)},>=Latex,scale=5, every node/.style={font=\footnotesize}]
  \draw[->,line width=0.35pt] (0,0,0) -- (0.86,0,0) node[right,font=\small] {$f_1$};
  \draw[->,line width=0.35pt] (0,0,0) -- (0,0.84,0) node[above right,font=\small] {$f_2$};
  \draw[->,line width=0.35pt] (0,0,0) -- (0,0,0.86) node[above,font=\small] {$f_3$};
  \fill[black] (0,0,0) circle (0.7pt);
  \node[below left] at (0,0,0) {$z^+$};

  \coordinate (p1) at (0.12,0.68,0.64);
  \coordinate (p2) at (0.21,0.54,0.52);
  \coordinate (p3) at (0.35,0.41,0.40);
  \coordinate (p4) at (0.52,0.29,0.28);
  \coordinate (p5) at (0.70,0.18,0.20);
  \coordinate (p6) at (0.27,0.30,0.62);
  \coordinate (p7) at (0.45,0.22,0.45);
  \coordinate (p8) at (0.37,0.50,0.25);


  \foreach \P/\Px/\Py/\lab/\col/\anch in {
    p1/0.12/0.68/$p_1$/violet!80!black/east,
    p2/0.21/0.54/$p_2$/blue!70!black/east,
    p3/0.35/0.41/$p_3$/orange!90!black/south,
    p4/0.52/0.29/$p_4$/green!60!black/west,
    p5/0.70/0.18/$p_5$/gray!80!black/west}{
    \fill[\col] (\P) circle (0.3pt);
    \draw[\col!70,densely dashed,line width=0.5pt] (\P)--(\Px,\Py,0);
    \node[inner sep=1pt,anchor=\anch] at (\P) {\lab};
  }
\end{tikzpicture}
\par\smallskip
{\small\textbf{(c)} Three-objective Pareto-front approximation\par}
\end{minipage}\hfill
\begin{minipage}{0.485\textwidth}
\centering
\resizebox{\linewidth}{!}{%
\input{triangular_mixture_plot_body.tex}%
}
\par\smallskip
{\small\textbf{(d)} Three-objective Tchebycheff-shadow surface\par}
\end{minipage}
\caption{Tchebycheff-shadow interpretation of the integral $\Rtwo$ indicator. Panels (a) and (b) reproduce the biobjective viewpoint used in the companion two-objective report: a finite Pareto-front approximation induces one shadow function per point over the weight interval, and the integral $\Rtwo$ value is the area under the lower envelope. Panels (c) and (d) show the analogous three-objective picture. A finite Pareto-front approximation now lives in objective space $\mathbb{R}^3$, the weight domain is the simplex $\Delta_2$, and the lower envelope becomes a piecewise-linear surface over that simplex. Panel (d) presents a triangular mixture plot of this surface: the color gradient indicates the height $\tau_S(w)$, contour lines visualize level sets, and the marked points $\kappa_1,\kappa_2,\kappa_3$ indicate representative kink points where the active Tchebycheff shadow changes. The integral $\Rtwo$ value is the volume under this surface over the simplex.}
\label{fig:tchebycheff-shadow-2d-3d}
\end{figure}

The two-objective case has substantially more order structure than the three-objective case.  A biobjective Pareto-front approximation can be sorted in one coordinate, after which the other coordinate is sorted in the reverse order.  This one-dimensional order yields an adjacent-neighbor decomposition of the integral $\Rtwo$ value and leads to exact polynomial-time dynamic-programming and matrix-search algorithms for fixed-cardinality subset selection~\cite{emmerich2026biobjective}.  In three objectives, the weight domain is the triangle $\DeltaTwo$ and the lower envelope has a two-dimensional cell structure.  The ordered-predecessor state used in the biobjective algorithms no longer applies, and the exact complexity of subset selection must be reconsidered.

The kink points highlighted in Figure~\ref{fig:tchebycheff-shadow-2d-3d}(d) are geometric signatures of the piecewise-linear envelope: they occur where the active supporting Tchebycheff shadow changes, and they help visualize why the three-objective envelope has a genuinely two-dimensional cell structure.

This paper gives a compact algorithmic picture for the three-objective problem.  The first contribution is an approximation result.  We show that, for any fixed baseline set $B$, the integral $\Rtwo$ improvement
\[
 I_B(S)=\Rtwo(B)-\Rtwo(B\cup S)
\]
is a monotone submodular set function.  This follows from the Tchebycheff-shadow representation: for every weight vector, the improvement is a maximum of nonnegative gains over the selected points, and integration preserves submodularity.  The standard greedy algorithm therefore obtains the classical $(1-1/e)$ guarantee for maximizing $I_B$ under a cardinality constraint.  For the ordinary ideal-point based $\Rtwo$ indicator, with the ideal point and weight distribution fixed, Theorem~\ref{thm:ideal-r2-greedy-gap} states the same guarantee directly in terms of the original $\Rtwo$ values: greedy closes at least a $(1-1/e)$-fraction of the maximum possible gap between a fixed dominated anchor value and the best cardinality-$k$ $\Rtwo$ value.  Equivalently, the residual gap to the optimum is at most $1/e$ of this initial anchor-to-optimum gap.

The second contribution is a hardness result.  We prove that exact fixed-cardinality integral-$\Rtwo$ subset selection is NP-hard already in three objectives.  The proof uses a perspective transformation from scalarization space to reciprocal objective space.  Under this transformation, Tchebycheff-shadow improvements become weighted anchored-box gains with density $(x_1+x_2+x_3)^{-4}$.  We then adapt the three-dimensional anchored-box construction of Bringmann, Cabello, and Emmerich~\cite{bringmann2018anchored}, using a triangular-grid independent-set gadget and showing that the required equal-gain and conflict-gap properties survive under the weighted measure.

The remainder of the paper is organized as follows.  Section~\ref{sec:related-work} discusses related work and the research gap addressed here.  Section~\ref{sec:preliminaries} defines integral $\Rtwo$ values, improvements, and Tchebycheff shadows.  Section~\ref{sec:submodularity} proves monotone submodularity, the greedy approximation guarantee, and its ideal-point based $\Rtwo$ gap-reduction form.  Section~\ref{sec:perspective} introduces the perspective representation by weighted anchored boxes.  Section~\ref{sec:hardness} proves NP-hardness in three objectives.  Section~\ref{sec:evaluation} discusses exact evaluation by arrangements in fixed dimension three, and Section~\ref{sec:algorithms} summarizes algorithmic consequences and practical solution directions.

\section{Related work and research gap}\label{sec:related-work}

The $\Rtwo$ indicator was introduced as a scalarization-based quality indicator for approximating the nondominated set and has often been used through a finite sample of weight vectors.  Recent work has clarified the importance of the continuous, exact, or integral version.  Sch\"apermeier and Kerschke showed that integrating over a continuum of Tchebycheff scalarizing functions yields a Pareto-compliant integral $\Rtwo$ indicator in the biobjective case~\cite{schaepermeier2024integral,schaepermeier2025integral}.  The biobjective subset-selection problem for this exact integral indicator is solved in~\cite{emmerich2026biobjective}, where the one-dimensional weight interval and the sorted Pareto-front order lead to exact dynamic programming and matrix-search algorithms.

Jaszkiewicz and Zielniewicz~\cite{jaszkiewicz2025exact} study exact calculation and properties of the $\Rtwo$ multiobjective quality indicator in general dimension.  They prove strict monotonicity when the indicator is calculated exactly and the reference point strongly dominates the evaluated set, propose an $L_\infty$ normalization of weight vectors, and adapt the improved quick hypervolume algorithm to obtain a Quick $\Rtwo$ (QR2) exact-calculation algorithm.  Their computation is based on diagonal lines of weighted Chebycheff functions, intersections with the upper boundary of the dominated region, and decompositions into hypercuboids.  Although the terminology of Tchebycheff shadows is not used there, this diagonal-line and boundary-intersection view is closely related to the lower-envelope/shadow representation used here: both describe the same continuous scalarization envelope, but from different geometric coordinates.

The research gap addressed in the present report is the tractability of the subset selection problem.  Whereas exact-calculation algorithms answer the question ``what is the integral $\Rtwo$ value of a given finite set?''  The fixed-cardinality subset-selection problem asks the optimization question ``which $k$ points should be retained?''  The two-dimensional report~\cite{emmerich2026biobjective} shows that this optimization problem is polynomial-time solvable in the biobjective case because the lower envelope has a one-dimensional order.  The present report shows that this tractability does not extend to three objectives: exact subset selection becomes NP-hard.  At the same time, the Tchebycheff-shadow representation shows that the improvement objective is monotone submodular, yielding a principled greedy approximation guarantee.  Thus the report connects the exact-calculation literature on integral $\Rtwo$ with the algorithmic questions of hardness and approximation for fixed-cardinality subset selection.

\section{Preliminaries and problem setting}\label{sec:preliminaries}

We use a minimization convention.  The utopian point is translated to the origin and all objective vectors are assumed to lie in $\RR^3_{>0}$.  A finite Pareto-front approximation is denoted by
\[
P=\{p^1,\ldots,p^n\}\subset \RR^3_{>0}.
\]
The notation does not require that all points are mutually nondominated, but this is the intended case.

\begin{definition}[Weighted Tchebycheff scalarization]
For $p=(p_1,p_2,p_3)\in\RR^3_{>0}$ and $w\in\DeltaTwo$, define
\[
 g_p(w)=\max\{w_1p_1,w_2p_2,w_3p_3\}.
\]
For a nonempty finite set $S\subseteq P$, define its Tchebycheff envelope by
\[
 \tau_S(w)=\min_{p\in S}g_p(w).
\]
\end{definition}

\begin{definition}[Integral $\Rtwo$ value]
The unnormalized integral $\Rtwo$ value of a nonempty set $S$ is
\[
 \Rtwo(S)=\int_{\DeltaTwo}\tau_S(w)\dd w.
\]
Here $\dd w$ denotes two-dimensional Lebesgue measure on the simplex.  A normalized convention divides by $\vol(\DeltaTwo)$; this positive constant has no effect on subset selection or approximation ratios.
\end{definition}

\begin{problem}[Three-objective integral $\Rtwo$ subset selection]
Given $P\subset\RR^3_{>0}$ and an integer $k$, find
\[
 S^\star\in\argmin\{\Rtwo(S):S\subseteq P,\ |S|=k\}.
\]
\end{problem}

\subsection{Baseline improvements}

For approximation and submodularity results it is convenient to work with improvements relative to a fixed baseline set $B$.  The baseline can be a current incumbent approximation, or a single dominated anchor point $r$ that is worse than all candidate points.  If $r_i\geq p_i$ for all candidates $p\in P$ and all coordinates $i$, then $\tau_{\{r\}\cup S}=\tau_S$ for every nonempty $S$.  Hence minimizing $\Rtwo(S)$ is equivalent to maximizing $\Rtwo(\{r\})-\Rtwo(\{r\}\cup S)$.

\begin{definition}[Integral $\Rtwo$ improvement]
For a fixed baseline set $B$ and a candidate subset $S\subseteq P$, define
\[
 I_B(S)=\Rtwo(B)-\Rtwo(B\cup S).
\]
\end{definition}

\begin{definition}[Tchebycheff shadow]
For a baseline $B$ and a candidate point $p\in P$, define the Tchebycheff shadow of $p$ below the baseline envelope as
\[
 E_p^B=\{(w,t):w\in\DeltaTwo,\ g_p(w)\leq t<\tau_B(w)\}.
\]
If $g_p(w)\geq \tau_B(w)$, the vertical fibre above $w$ is empty.
\end{definition}

\begin{theorem}[Scalarization-space volume representation]\label{thm:shadow-volume}
For every finite $S\subseteq P$,
\[
 I_B(S)=\vol\left(\bigcup_{p\in S}E_p^B\right),
\]
where the volume is taken in $\DeltaTwo\times\RR$.
\end{theorem}

\begin{proof}
For a fixed weight vector $w$, the baseline scalarization value is $\tau_B(w)$ and the new envelope is
\[
 \tau_{B\cup S}(w)=\min\left\{\tau_B(w),\min_{p\in S}g_p(w)\right\}.
\]
The pointwise improvement is therefore
\[
 \tau_B(w)-\tau_{B\cup S}(w)
 =\max_{p\in S}\bigl(\tau_B(w)-g_p(w)\bigr)_+.
\]
The right-hand side is precisely the length of the union of the vertical intervals
\[
 [g_p(w),\tau_B(w))\quad\text{for }p\in S
\]
that are nonempty.  Integrating these fibre lengths over $w\in\DeltaTwo$ gives the volume of the union of the shadows.
\end{proof}

Theorem~\ref{thm:shadow-volume} is the geometric bridge used throughout the report.  It is analogous to dominated-volume representations of the hypervolume indicator, but the volume now lives in scalarization space rather than directly as ordinary objective-space dominated volume.

\section{Submodularity and greedy approximation}\label{sec:submodularity}

The shadow representation immediately gives a positive algorithmic result.  Written as an improvement problem, fixed-cardinality integral-$\Rtwo$ subset selection is an instance of monotone submodular maximization.

\begin{lemma}[Pointwise coverage form]\label{lem:pointwise-submodular}
For each fixed $w\in\DeltaTwo$, the function
\[
 F_w(S)=\max_{p\in S}\bigl(\tau_B(w)-g_p(w)\bigr)_+
\]
is monotone nondecreasing and submodular in $S$.
\end{lemma}

\begin{proof}
Let $a_p(w)=(\tau_B(w)-g_p(w))_+\geq 0$.  Then $F_w(S)=\max_{p\in S}a_p(w)$, with $F_w(\emptyset)=0$.  Monotonicity is immediate.  For submodularity, take $A\subseteq C\subseteq P$ and $q\notin C$.  The marginal gain is
\[
 F_w(A\cup\{q\})-F_w(A)=\max\{0,a_q(w)-F_w(A)\}.
\]
Since $F_w(A)\leq F_w(C)$, this marginal gain is at least
\[
 \max\{0,a_q(w)-F_w(C)\}=F_w(C\cup\{q\})-F_w(C).
\]
This is the diminishing-returns condition.
\end{proof}

\begin{theorem}[Submodularity of integral $\Rtwo$ improvement]\label{thm:submodularity}
For every fixed baseline set $B$, the set function $I_B:2^P\to\RR_{\geq 0}$ is monotone nondecreasing and submodular.
\end{theorem}

\begin{proof}
By the scalarization-space volume representation,
\[
 I_B(S)=\int_{\DeltaTwo}F_w(S)\dd w.
\]
A nonnegative linear combination, and therefore an integral, of monotone submodular functions is monotone submodular.
\end{proof}

\begin{corollary}[Supermodularity of the raw integral $\Rtwo$ value]
The raw value $S\mapsto \Rtwo(B\cup S)$ is monotone nonincreasing and supermodular.
\end{corollary}

\begin{proof}
The identity $\Rtwo(B\cup S)=\Rtwo(B)-I_B(S)$ converts monotone submodularity of $I_B$ into monotone nonincreasing supermodularity.
\end{proof}

\begin{theorem}[Greedy approximation]\label{thm:greedy}
Let $S_g$ be the set returned by the greedy algorithm that starts with $S_0=\emptyset$ and repeatedly adds a point of maximum marginal improvement
\[
 p_t\in\argmax_{p\in P\setminus S_{t-1}}\left(I_B(S_{t-1}\cup\{p\})-I_B(S_{t-1})\right)
\]
until $|S_g|=k$.  Then
\[
 I_B(S_g)\geq \left(1-\frac{1}{e}\right)\max_{|S|\leq k}I_B(S).
\]
\end{theorem}

\begin{proof}
This is the classical Nemhauser--Wolsey--Fisher theorem for maximizing a monotone submodular function under a cardinality constraint~\cite{nemhauser1978analysis}.  Theorem~\ref{thm:submodularity} verifies the hypotheses.
\end{proof}

\subsection{Ideal-point based \texorpdfstring{$\Rtwo$}{R2} gap guarantee}\label{subsec:ideal-r2-guarantee}

The preceding theorem is stated for a general baseline set.  For the original subset-selection problem it is useful to express the same bound directly in terms of the ideal-point based $\Rtwo$ value.  Throughout this paper the ideal point is fixed and translated to the origin, so the scalarization functions are $g_p(w)=\max_i w_i p_i$.  If the ideal point is denoted explicitly by $z^{\mathrm{id}}$, then $p_i$ should be read as the shifted loss $f_i(p)-z_i^{\mathrm{id}}$.  The fixed dominated anchor in the next theorem is only a normalization level for converting the decreasing $\Rtwo$ value into a nonnegative gain; it is not a baseline approximation set.

\begin{theorem}[Greedy gap guarantee for ideal-point based $\Rtwo$]\label{thm:ideal-r2-greedy-gap}
Let $P\subset\RR^d_{>0}$ be a finite candidate set.  Fix the ideal point, the weight domain, and the weight distribution used in the Tchebycheff losses.  Let $\bar p\in\RR^d_{>0}$ be a fixed dominated anchor satisfying
\[
 g_{\bar p}(w)\geq g_p(w)\qquad\text{for all }p\in P\text{ and all weights }w.
\]
For every nonempty $S\subseteq P$, define the ideal-point based $\Rtwo$ value by
\[
 \Rtwo(S)=\int_{\Delta_{d-1}}\min_{p\in S}g_p(w)\dd w,
 \qquad
 g_p(w)=\max_{i=1,\ldots,d}w_ip_i,
\]
and set
\[
 R_0:=\Rtwo(\{\bar p\}),
 \qquad
 Q_{\bar p}(S):=R_0-\Rtwo(S),\qquad S\neq\emptyset,
\]
with $Q_{\bar p}(\emptyset)=0$.  Let $G_k$ be the subset obtained by the greedy algorithm that starts from the empty set and, at each step, adds a point maximizing the marginal increase of $Q_{\bar p}$.  Let
\[
 S_k^\star\in\argmin\{\Rtwo(S):S\subseteq P,\ |S|=k\}
\]
be an optimal cardinality-$k$ subset.  Then
\[
 R_0-\Rtwo(G_k)
 \geq
 \left(1-\frac1e\right)\bigl(R_0-\Rtwo(S_k^\star)\bigr).
\]
Equivalently, if $R_0>\Rtwo(S_k^\star)$,
\[
 \frac{R_0-\Rtwo(G_k)}{R_0-\Rtwo(S_k^\star)}
 \geq 1-\frac1e.
\]
In residual-gap form,
\[
 \Rtwo(G_k)-\Rtwo(S_k^\star)
 \leq
 \frac1e\bigl(R_0-\Rtwo(S_k^\star)\bigr),
\]
and equivalently,
\[
 \Rtwo(G_k)
 \leq
 \left(1-\frac1e\right)\Rtwo(S_k^\star)+\frac1e R_0.
\]
Thus, for a fixed ideal point and fixed weight distribution, greedy closes at least a $(1-1/e)$-fraction of the maximum possible $\Rtwo$ gap between the fixed anchor value $R_0$ and the optimal cardinality-$k$ value.
\end{theorem}

\begin{proof}
For each fixed weight vector $w$, define
\[
 q_w(S)=g_{\bar p}(w)-\min_{p\in S}g_p(w),\qquad S\neq\emptyset,
\]
and set $q_w(\emptyset)=0$.  Since $g_{\bar p}(w)\geq g_p(w)$ for all candidates, we have
\[
 q_w(S)=\max_{p\in S}\bigl(g_{\bar p}(w)-g_p(w)\bigr).
\]
This is a maximum of nonnegative singleton gains.  It is monotone, and its marginal gain from adding a new point can only decrease as the selected set grows.  Hence $q_w$ is monotone submodular.  Integration over the fixed weight domain preserves monotonicity and submodularity, so
\[
 Q_{\bar p}(S)=\int q_w(S)\dd w=R_0-\Rtwo(S)
\]
is a nonnegative monotone submodular set function.  The Nemhauser--Wolsey--Fisher greedy theorem for cardinality-constrained monotone submodular maximization gives
\[
 Q_{\bar p}(G_k)\geq \left(1-\frac1e\right)Q_{\bar p}(S_k^\star).
\]
Substituting $Q_{\bar p}(S)=R_0-\Rtwo(S)$ gives the displayed equivalent forms.
\end{proof}

\begin{remark}[Why the normalization is needed]
The shift by the fixed anchor value $R_0$ is the natural normalization that turns the decreasing ideal-point based $\Rtwo$ objective into a nonnegative monotone gain.  The guarantee is therefore not a multiplicative approximation of the raw residual value $\Rtwo(S_k^\star)$ itself.  It is a gap-reduction guarantee for the original $\Rtwo$ values: the greedy residual gap to the optimal cardinality-$k$ value is at most $1/e$ of the initial anchor-to-optimum gap.  No universal bound on the ratio $\Rtwo(G_k)/\Rtwo(S_k^\star)$ follows from submodularity alone without an additional bound on $R_0/\Rtwo(S_k^\star)$.  The ideal point, the anchor, and the weight distribution must remain fixed throughout the selection process.  If the ideal point is recomputed for every selected subset, the scalarization functions of all points change and this submodularity argument no longer applies directly.
\end{remark}

\begin{pseudocode}{Greedy integral-$\Rtwo$ subset selection}
\begin{tabular}{ll}
\textbf{Input:} & candidate set $P$, baseline $B$, cardinality $k$ \\
\textbf{Output:} & subset $S\subseteq P$, $|S|=k$ \\
1. & $S\leftarrow\emptyset$ \\
2. & \textbf{for} $t=1,\ldots,k$ \textbf{do} \\
3. & \quad choose $p\in P\setminus S$ maximizing $I_B(S\cup\{p\})-I_B(S)$ \\
4. & \quad $S\leftarrow S\cup\{p\}$ \\
5. & \textbf{return} $S$.
\end{tabular}
\caption{Greedy selection is meaningful in any dimension once integral $\Rtwo$ is written as a fixed-envelope improvement.  With a dominated anchor $\bar p$ and $R_0=\Rtwo(\{\bar p\})$, the same iterations maximize the normalized ideal-point based gain $Q_{\bar p}(S)=R_0-\Rtwo(S)$.  Exact or approximate marginal evaluation can be supplied by quadrature, arrangement integration, or Monte Carlo estimation.}
\end{pseudocode}

\begin{remark}[Approximation target]
The approximation ratio applies to an attained reduction in $\Rtwo$, or equivalently to a normalized gain.  Theorem~\ref{thm:ideal-r2-greedy-gap} states this in terms of the ideal-point based $\Rtwo$ values themselves: greedy achieves a $(1-1/e)$-fraction of the maximum possible reduction from the fixed anchor value $R_0$ to the best cardinality-$k$ value, or equivalently leaves at most a $1/e$ residual anchor-to-optimum gap.  This should not be read as a multiplicative approximation of the small residual value $\Rtwo(S_k^\star)$ itself.
\end{remark}

\section{A perspective representation by weighted anchored boxes}\label{sec:perspective}

The NP-hardness proof uses a second geometric representation.  It maps the scalarization-space integral into a weighted anchored-box problem in reciprocal coordinates.  This is the point where the integral $\Rtwo$ geometry becomes comparable to the anchored-box formulation of hypervolume subset selection.

For $a=(a_1,a_2,a_3)\in\RR^3_{>0}$, let
\[
 \boxop(a)=[0,a_1]\times[0,a_2]\times[0,a_3].
\]
For a finite set $A\subset\RR^3_{>0}$, write
\[
 U(A)=\bigcup_{a\in A}\boxop(a).
\]
Define the density
\[
 \varphi(x)=\frac{1}{(x_1+x_2+x_3)^4},\qquad x\in\RR^3_{>0}.
\]
Although $\varphi$ is not integrable near the origin over arbitrary neighborhoods, all relative gain regions used below are bounded away from the origin and have finite measure.

\begin{theorem}[Perspective representation]\label{thm:perspective}
Let $S\subset\RR^3_{>0}$ be finite and nonempty.  For each objective vector $p\in S$, define the reciprocal box corner
\[
 b(p)=\left(\frac1{p_1},\frac1{p_2},\frac1{p_3}\right).
\]
Then
\[
 \Rtwo(S)
 =
 \int_{\RR^3_{>0}\setminus U(b(S))}
 \frac{\dd x}{(x_1+x_2+x_3)^4},
\]
where $b(S)=\{b(p):p\in S\}$.
Consequently, for a baseline set $B$ and candidate set $S$,
\[
 I_B(S)=
 \int_{U(b(B\cup S))\setminus U(b(B))}
 \frac{\dd x}{(x_1+x_2+x_3)^4}.
\]
\end{theorem}

\begin{proof}
Write $\tau_S(w)=\min_{p\in S}g_p(w)$.  Since $\tau_S(w)>0$ on $\DeltaTwo$,
\[
 \Rtwo(S)=\int_{\DeltaTwo}\int_0^{\tau_S(w)}\dd t\dd w.
\]
Apply the change of variables
\[
 x_i=\frac{w_i}{t},\qquad i=1,2,3.
\]
Conversely,
\[
 t=\frac{1}{x_1+x_2+x_3},
 \qquad
 w_i=\frac{x_i}{x_1+x_2+x_3}.
\]
For a fixed point $p$, the inequality $t\leq g_p(w)$ is equivalent to
\[
 \frac{1}{x_1+x_2+x_3}
 \leq
 \max_i\frac{x_i}{x_1+x_2+x_3}p_i,
\]
which is equivalent to $1\leq\max_i x_ip_i$.  This means that $x\notin\boxop(b(p))$.  Thus $t\leq\tau_S(w)$ holds if and only if $x\notin U(b(S))$.

The absolute Jacobian determinant of the map $(w_1,w_2,t)\mapsto(x_1,x_2,x_3)$ is $t^{-4}$.  Hence
\[
 \dd w\dd t=t^4\dd x=\frac{\dd x}{(x_1+x_2+x_3)^4}.
\]
Substitution gives the first identity.  The second identity follows by subtracting the representation for $B\cup S$ from the representation for $B$.
\end{proof}

\begin{figure}[t]
\centering
\begin{tikzpicture}[>=Latex,scale=1]
  \node[draw,rounded corners,fill=blue!5,minimum width=5.2cm,minimum height=2.0cm,align=center] (left) at (0,0) {scalarization space\\$w\in\Delta_2$, height $t$\\Tchebycheff shadow};
  \node[draw,rounded corners,fill=orange!8,minimum width=5.2cm,minimum height=2.0cm,align=center] (right) at (9.0,0) {reciprocal space\\$x_i=w_i/t$\\weighted anchored-box gain};
  \draw[->,thick] (left) -- node[above,align=center] {perspective map\\$x_i=w_i/t$} (right);
  \node[align=center,font=\small] at (3.5,-1.55) {$\displaystyle \dd w\dd t=(x_1+x_2+x_3)^{-4}\dd x$};
\end{tikzpicture}
\caption{The perspective transformation converts Tchebycheff-shadow improvement into weighted anchored-box gain.  This is the key bridge from integral $\Rtwo$ to the three-dimensional anchored-box hardness construction.}
\label{fig:perspective}
\end{figure}

\begin{figure}[t]
\centering
\resizebox{\textwidth}{!}{\input{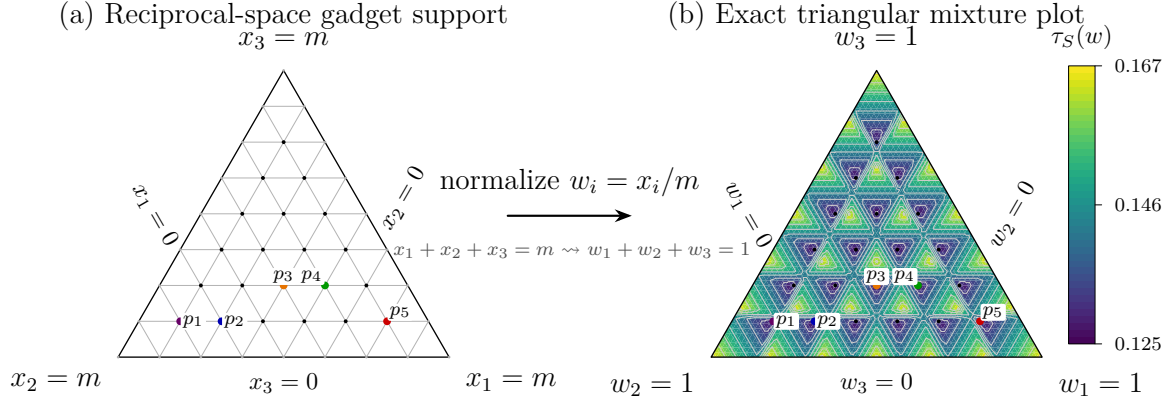}}
\caption{The common triangular-grid gadget in the two coordinate systems used by the hardness proof. Left: a front view of the layer $x_1+x_2+x_3=m$ in reciprocal objective space. Boundary lattice points are gray; positive lattice points are black and correspond to finite reciprocal objective vectors. Right: the same positive lattice points after normalization $w_i=x_i/m$, shown in the simplex $\Delta_2$. The color field is the exact lower envelope of the corresponding Tchebycheff shadows, $\tau_S(w)=\min_{x\in P_m^+}\max_i w_i/x_i$, where $P_m^+=\{x\in\mathbb{N}_{>0}^3:x_1+x_2+x_3=m\}$. The highlighted points $p_1,\ldots,p_5$ are identical before and after the mapping.}
\label{fig:hv-r2-gadget}
\end{figure}

Figure~\ref{fig:hv-r2-gadget} makes explicit the normalization step used to compare the anchored-box lattice gadget with the Tchebycheff-shadow geometry: the same positive lattice points that support the reciprocal-space construction generate the exact triangular mixture plot on the simplex.  Appendix~\ref{app:perspective-verification} gives a short Python verification experiment for the pointwise map, the Jacobian factor, and the resulting weighted-box integral identity.

\section{NP-hardness in three objectives}\label{sec:hardness}

We now prove NP-hardness of exact fixed-cardinality integral-$\Rtwo$ subset selection in three objectives.  The proof follows the anchored-box construction of Bringmann, Cabello, and Emmerich~\cite{bringmann2018anchored}, with one additional observation: the density $\varphi(x)=(x_1+x_2+x_3)^{-4}$ is constant under translations that preserve the coordinate sum.  The difference regions in the anchored-box gadget are translates, in such sum-preserving directions, of a fixed finite union of boxes.  Hence the equal-contribution and equal-conflict-gap arguments survive under the weighted measure.

\subsection{The source problem}

Let $\Gamma$ be the infinite triangular grid.  The vertex set can be represented as
\[
 V(\Gamma)=\left\{\left(i+\frac{j}{2},\frac{\sqrt{3}}{2}j\right):i,j\in\mathbb{N}\right\},
\]
with edges between vertices at Euclidean distance one.  The following problem is NP-complete; it is used as the source problem in the three-dimensional anchored-box reduction~\cite{bringmann2018anchored}.  More specifically, Bringmann, Cabello, and Emmerich reduce to this problem from Independent Set on planar graphs of maximum degree at most three.  Their argument uses the complementarity with Vertex Cover on planar graphs of degree at most three~\cite{garey1977rectilinear}, then embeds such graphs in a polynomial-size square grid using standard planar-grid drawing results, for example Storer~\cite{storer1984minimal} or Tamassia and Tollis~\cite{tamassia1989planar}.  A shear map turns the square grid into a subgraph of the triangular grid, after which local scaling and rerouting make the representation induced and give the required path parities.

\begin{problem}[Independent Set on Induced Triangular Grid]
Given a finite induced subgraph $A$ of $\Gamma$ and an integer $\ell$, determine whether $A$ has an independent set of size $\ell$.
\end{problem}

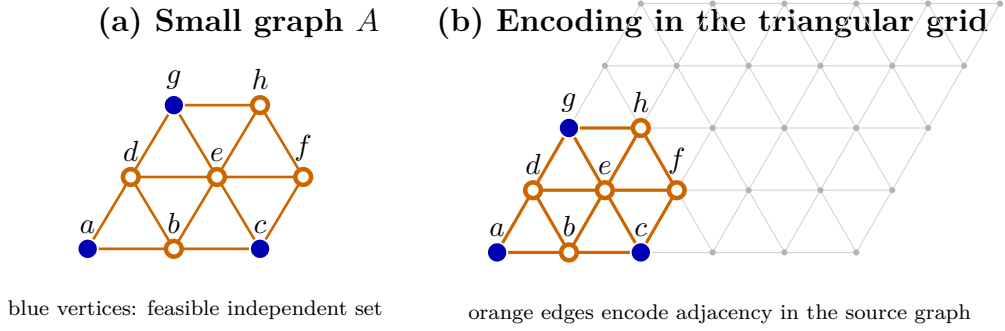
\begin{figure}[t]
\centering
\begin{tikzpicture}[scale=0.95, every node/.style={font=\small}]
  \begin{scope}[shift={(0,0)}]
    \node[font=\bfseries] at (2.1,3.15) {(a) Small graph $A$};
    \coordinate (a) at (0.0,0.0);
    \coordinate (b) at (1.2,0.0);
    \coordinate (c) at (2.4,0.0);
    \coordinate (d) at (0.6,1.0);
    \coordinate (e) at (1.8,1.0);
    \coordinate (f) at (3.0,1.0);
    \coordinate (g) at (1.2,2.0);
    \coordinate (h) at (2.4,2.0);
    \foreach \u/\v in {a/b,b/c,d/e,e/f,g/h,a/d,b/d,b/e,c/e,c/f,d/g,e/g,e/h,f/h}{%
      \draw[orange!80!black,line width=1.0pt] (\u)--(\v);
    }
    \foreach \v/\lab in {a/$a$,b/$b$,c/$c$,d/$d$,e/$e$,f/$f$,g/$g$,h/$h$}{%
      \fill[orange!80!black] (\v) circle (4.0pt);
      \fill[white] (\v) circle (2.2pt);
      \node[above=2pt] at (\v) {\lab};
    }
    \foreach \v in {a,c,g}{%
      \fill[blue!70!black] (\v) circle (4.0pt);
      \draw[white,line width=0.7pt] (\v) circle (4.0pt);
    }
    \node[align=center,font=\scriptsize,fill=white,rounded corners=2pt,inner sep=2pt] at (1.5,-0.85)
      {blue vertices: feasible independent set};
  \end{scope}

  \begin{scope}[shift={(5.7,-0.05)}]
    \node[font=\bfseries] at (3.0,3.20) {(b) Encoding in the triangular grid};
    \foreach \j in {0,...,4}{%
      \foreach \i in {0,...,5}{%
        \coordinate (v\i-\j) at ({\i+0.5*\j},{0.866*\j});
      }
    }
    \foreach \j in {0,...,4}{%
      \foreach \i in {0,...,4}{\draw[gray!35] (v\i-\j)--(v\the\numexpr\i+1\relax-\j);} }
    \foreach \j in {0,...,3}{%
      \foreach \i in {0,...,5}{\draw[gray!35] (v\i-\j)--(v\i-\the\numexpr\j+1\relax);} }
    \foreach \j in {0,...,3}{%
      \foreach \i in {0,...,4}{\draw[gray!35] (v\the\numexpr\i+1\relax-\j)--(v\i-\the\numexpr\j+1\relax);} }
    \foreach \j in {0,...,4}{%
      \foreach \i in {0,...,5}{\fill[gray!60] (v\i-\j) circle (1.2pt);} }

    \coordinate (A) at (v0-0);
    \coordinate (B) at (v1-0);
    \coordinate (C) at (v2-0);
    \coordinate (D) at (v0-1);
    \coordinate (E) at (v1-1);
    \coordinate (F) at (v2-1);
    \coordinate (G) at (v0-2);
    \coordinate (H) at (v1-2);

    \foreach \u/\v in {A/B,B/C,D/E,E/F,G/H,A/D,B/D,B/E,C/E,C/F,D/G,E/G,E/H,F/H}{%
      \draw[orange!80!black,line width=1.2pt] (\u)--(\v);
    }

    \foreach \v/\lab in {A/$a$,B/$b$,C/$c$,D/$d$,E/$e$,F/$f$,G/$g$,H/$h$}{%
      \fill[orange!80!black] (\v) circle (4.0pt);
      \fill[white] (\v) circle (2.2pt);
      \node[above=2pt] at (\v) {\lab};
    }
    \foreach \v in {A,C,G}{%
      \fill[blue!70!black] (\v) circle (4.0pt);
      \draw[white,line width=0.7pt] (\v) circle (4.0pt);
    }

    \node[align=center,font=\scriptsize,fill=white,rounded corners=2pt,inner sep=2pt] at (3.1,-0.85)
      {orange edges encode adjacency in the source graph};
  \end{scope}
\end{tikzpicture}
\caption{A small induced-subgraph example for the hardness reduction, in the spirit of the triangular-grid construction of Bringmann, Cabello, and Emmerich. Left: the abstract source graph $A$, with vertices and edges shown explicitly. Right: the same graph embedded as an induced subgraph of the triangular grid $\Gamma$. The blue vertices illustrate a feasible  independent set, while the orange edges indicate those adjacencies whose conflicts are encoded by intersections of optional gain regions in the anchored-box construction.  In panel (b), every orange edge is a genuine triangular-grid edge.}
\label{fig:triangular-grid}
\end{figure}

\subsection{The anchored-box scaffold}

For an integer $m\geq 3$, define
\[
 P_m=\{(x,y,z)\in\mathbb{N}^3:x+y+z=m\}.
\]
These points form a triangular layer in the positive integer lattice.  Let
\[
 \eps=\frac{1}{100m^4}
\]
and, for each $p\in P_{m-1}$, define
\[
 q(p)=p+(\eps,\eps,\eps).
\]
Let
\[
 Q_m=\{q(p):p\in P_{m-1}\}.
\]
For $q=q(p)$, define the relative gain region
\[
 D(q)=U(P_m\cup\{q\})\setminus U(P_m).
\]
All these regions are bounded away from the origin, so their $\varphi$-weighted measure is finite.  We write
\[
 \muvarphi(X)=\int_X \varphi(x)\dd x
\]
for such bounded measurable sets $X$.

\begin{lemma}[Shape of an optional gain region]\label{lem:diff-shape}
Let $p=(p_1,p_2,p_3)\in P_{m-1}$ and $q=q(p)$.  Up to boundary sets of measure zero,
\[
D(q)=C(p)\cup L_{12}(p)\cup L_{13}(p)\cup L_{23}(p),
\]
where
\[
 C(p)=\prod_{i=1}^3(p_i,p_i+\eps]
\]
is a cube of side length $\eps$, and
\[
\begin{aligned}
L_{12}(p)&=(p_1,p_1+\eps]\times(p_2,p_2+\eps]\times(p_3-1,p_3],\\
L_{13}(p)&=(p_1,p_1+\eps]\times(p_2-1,p_2]\times(p_3,p_3+\eps],\\
L_{23}(p)&=(p_1-1,p_1]\times(p_2,p_2+\eps]\times(p_3,p_3+\eps].
\end{aligned}
\]
Thus $D(q)$ is the union of three thin boxes of size $\eps\times\eps\times 1$ and one cube of size $\eps^3$.
\end{lemma}

\begin{proof}
Let $x\in\boxop(q)$.  A scaffold point $r\in P_m$ dominates $x$ if and only if there is an integer vector $r\geq x$ with $r_1+r_2+r_3=m$.  Since $q_i=p_i+\eps<p_i+1$, each coordinate of $x$ lies below $p_i+1$.  The point $x$ fails to be dominated by any scaffold point precisely when the componentwise integer ceiling of $x$ has coordinate sum at least $m+1$.  Since $p_1+p_2+p_3=m-1$, this can occur only when at least two coordinates lie in the short intervals $(p_i,p_i+\eps]$ and the remaining coordinate lies above $p_j-1$.  This gives exactly the three thin boxes listed above, together with their common triple-overlap cube $C(p)$.  Boundary differences have measure zero and do not affect the weighted measure.
\end{proof}

\begin{lemma}[Equal weighted contribution]\label{lem:equal-weighted}
There is a number $\delta_m>0$, depending only on $m$ and $\eps$, such that
\[
 \muvarphi(D(q))=\delta_m
\]
for every $q\in Q_m$.
\end{lemma}

\begin{proof}
By Lemma~\ref{lem:diff-shape}, each region $D(q(p))$ is obtained from any other region $D(q(p'))$ by translating the same finite union of boxes by the vector $p-p'$.  Since both $p$ and $p'$ lie in $P_{m-1}$, this translation vector has coordinate sum zero.  The density
\[
 \varphi(x)=\frac{1}{(x_1+x_2+x_3)^4}
\]
depends only on the coordinate sum.  Hence the translation preserves the density pointwise.  Therefore all regions have the same weighted measure.
\end{proof}

\begin{lemma}[Intersection graph]\label{lem:intersection-graph}
Define a graph $T_m$ on vertex set $Q_m$ by connecting $q,q'\in Q_m$ if and only if
\[
 D(q)\cap D(q')
\]
has positive volume.  Then $T_m$ is isomorphic to a finite induced subgraph of the triangular grid.  Moreover, there is a number $\rho_m>0$ such that for every edge $qq'$ of $T_m$,
\[
 \muvarphi(D(q)\cap D(q'))=\rho_m.
\]
\end{lemma}

\begin{proof}
Write $q=q(p)$ and $q'=q(p')$ with $p,p'\in P_{m-1}$.  By Lemma~\ref{lem:diff-shape}, an intersection can occur only when one of the short directions of $D(q)$ meets a long direction of $D(q')$, and symmetrically.  This happens exactly when
\[
 p'-p\in\{\pm(1,-1,0),\pm(1,0,-1),\pm(0,1,-1)\}.
\]
These are precisely the six adjacency directions of a triangular-grid layer.  In that case the intersection is, up to boundary, a cube of side length $\eps$.  For example, if $p'=p+(1,-1,0)$, then
\[
D(q)\cap D(q')=(p_1,p_1+\eps]\times(p_2-1,p_2-1+\eps]\times(p_3,p_3+\eps]
\]
up to measure-zero boundaries.  The other five cases are obtained by coordinate permutation.

All such intersection cubes are translates of one another by vectors whose coordinate sum is zero.  Since $\varphi$ depends only on $x_1+x_2+x_3$, they all have the same positive weighted measure.  Denote this measure by $\rho_m$.
\end{proof}

\begin{lemma}[Independent sets and weighted gain]\label{lem:weighted-gain}
For any $Q'\subseteq Q_m$:
\begin{enumerate}[label=(\alph*)]
\item if $Q'$ is independent in $T_m$, then
\[
 \muvarphi\bigl(U(P_m\cup Q')\setminus U(P_m)\bigr)=|Q'|\delta_m;
\]
\item if $Q'$ contains an edge of $T_m$, then
\[
 \muvarphi\bigl(U(P_m\cup Q')\setminus U(P_m)\bigr)
 \leq |Q'|\delta_m-\rho_m.
\]
\end{enumerate}
\end{lemma}

\begin{proof}
The relative gain region is
\[
 U(P_m\cup Q')\setminus U(P_m)=\bigcup_{q\in Q'}D(q).
\]
If $Q'$ is independent, the regions $D(q)$ are pairwise disjoint up to measure-zero boundaries by Lemma~\ref{lem:intersection-graph}, so the weighted measure is the sum of their identical weighted measures, namely $|Q'|\delta_m$.

If $Q'$ contains an edge $qq'$, then inclusion-exclusion gives
\[
 \muvarphi\left(\bigcup_{r\in Q'}D(r)\right)
 \leq \sum_{r\in Q'}\muvarphi(D(r))-\muvarphi(D(q)\cap D(q'))
 = |Q'|\delta_m-\rho_m.
\]
\end{proof}

\begin{lemma}[Scaffold priority]\label{lem:scaffold-priority}
For $m$ sufficiently large and $\eps=1/(100m^4)$, any selection that omits at least one scaffold point from $P_m$ has worse integral-$\Rtwo$ value than every selection containing all of $P_m$ and at least one optional point from $Q_m$.
Equivalently, in reciprocal anchored-box gain language, the total possible weighted gain from all optional regions is smaller than the weighted loss caused by omitting one scaffold box.
\end{lemma}

\begin{proof}
The construction is contained in the cube $[0,m]^3$.  The optional gain regions satisfy
\[
        m-2 < x_1+x_2+x_3 < m+1,
\]
up to boundary sets of measure zero.  Hence, for $m\geq 4$,
\[
        \varphi(x)\leq \frac{16}{m^4}
\]
on every optional gain region.  By Lemma~\ref{lem:diff-shape}, each optional region has ordinary volume $3\eps^2+\eps^3\leq 4\eps^2$, and $|Q_m|<m^2$.  Thus the total optional weighted gain is at most
\[
        m^2\cdot 4\eps^2\cdot \frac{16}{m^4}
        =\frac{64\eps^2}{m^2}.
\]
With $\eps=1/(100m^4)$ this is less than $1/(100m^4)$.

On the other hand, the ordinary anchored-box scaffold has the property used in the unweighted construction: omitting a point of $P_m$ removes an exclusive region of ordinary volume at least one from the scaffold union.  This exclusive region lies in $[0,m]^3$, where $x_1+x_2+x_3\leq m$ on the relevant boundary layer.  Therefore its weighted measure is at least $1/m^4$.  The total weighted gain of all optional boxes is smaller than this loss.  Hence no collection of optional boxes can compensate for omitting a scaffold point.
\end{proof}

\begin{remark}[On the choice of $\eps$]
The original unweighted anchored-box proof uses $\eps=1/(4m^2)$, because ordinary volume gives unit loss for a missing scaffold point and optional gain $O(m^2\eps^2)$.  For the weighted density, using a smaller rational $\eps=1/(100m^4)$ is a harmless strengthening that keeps the same intersection graph and gives a clean polynomial separation between mandatory scaffold loss and total optional gain.  All coordinates remain rational with polynomial bit length.
\end{remark}

\subsection{Reduction to integral \texorpdfstring{$\Rtwo$}{R2} subset selection}

We now formulate the decision problem.

\begin{problem}[Three-objective integral $\Rtwo$ subset selection]
Given a finite Pareto-front approximation $P\subset\RR^3_{>0}$, an integer $k$, and a rational threshold $T$, decide whether there exists $S\subseteq P$ with $|S|=k$ and
\[
 \Rtwo(S)\leq T.
\]
\end{problem}

\begin{theorem}[NP-hardness in three objectives]\label{thm:np-hard}
Three-objective fixed-cardinality integral-$\Rtwo$ subset selection is NP-hard.
\end{theorem}

\begin{proof}
We reduce from Independent Set on induced triangular-grid graphs.  Given an instance $(A,\ell)$, choose $m$ large enough so that the graph $T_m$ from Lemma~\ref{lem:intersection-graph} contains an induced copy of $A$.  This is possible because $T_m$ is a finite triangular-grid patch whose size grows with $m$.

Construct the anchored-box scaffold $P_m$ and the optional points $Q_m(A)\subseteq Q_m$ corresponding to the vertices of $A$.  The anchored-box candidate set is
\[
 C=P_m\cup Q_m(A).
\]
Now pass to objective space by reciprocal coordinates:
\[
 \widehat C=\left\{\left(\frac1{c_1},\frac1{c_2},\frac1{c_3}\right):c\in C\right\}\subset\RR^3_{>0}.
\]
Set
\[
 k=|P_m|+\ell.
\]
By the scaffold-priority lemma, any optimal or threshold-reaching selection must contain all reciprocal points corresponding to $P_m$.  Therefore the remaining choice is exactly a choice of $\ell$ optional points.

Let $C_0=P_m$ be the scaffold.  By the perspective representation, for any optional set $Q'\subseteq Q_m(A)$, the integral-$\Rtwo$ improvement over the scaffold is
\[
 \muvarphi\bigl(U(C_0\cup Q')\setminus U(C_0)\bigr).
\]
By Lemma~\ref{lem:weighted-gain}, this improvement is exactly $\ell\delta_m$ when the selected optional vertices form an independent set of size $\ell$, and it is at most $\ell\delta_m-\rho_m$ when the selected optional vertices contain an edge.  Choose a threshold $T$ halfway between the two corresponding $\Rtwo$ values, for instance
\[
 T=
 \Rtwo(\widehat C_0)-\ell\delta_m+\frac{\rho_m}{2},
\]
where $\widehat C_0$ denotes the reciprocal objective points corresponding to $C_0$.  Then there exists a size-$k$ subset of $\widehat C$ with $\Rtwo$ value at most $T$ if and only if $A$ has an independent set of size $\ell$.

All coordinates used in the construction are rational with polynomial bit length.  The quantities $\delta_m$ and $\rho_m$ are integrals of $(x_1+x_2+x_3)^{-4}$ over finite unions of boxes with rational endpoints.  Repeated elementary integration expresses them as rational combinations of reciprocal linear terms in the endpoints, hence with polynomially representable thresholds.  Therefore the reduction is a polynomial many-one reduction.
\end{proof}

\begin{remark}[Nondominance]
The anchored-box construction uses points on triangular layers and their small perturbations.  These points are nondominated in the maximization order of anchored boxes.  Taking reciprocals converts them into nondominated minimization points.  If a strict general-position Pareto-front approximation is desired, arbitrarily small rational perturbations can separate equal coordinates without changing the gap in the reduction.
\end{remark}

\section{Exact evaluation in three objectives}\label{sec:evaluation}

The approximation and hardness results are independent of a particular evaluation algorithm.  For practical algorithms, however, one needs to compute or estimate $I_B(S)$ and its marginal gains.  In three objectives, exact evaluation is possible by planar subdivision of the weight simplex\footnote{In \cite{jaszkiewicz2025exact}, an algorithm for integral $R_2$ computation that is exponential in $n$ but fast on average is proposed.}.  Appendix~\ref{app:implementation} gives a compact reference implementation of this arrangement-based evaluator and validates it against Monte Carlo integration.

\begin{theorem}[Arrangement-based exact evaluation in fixed dimension three]\label{thm:arrangement}
Let $S\subseteq P$ and assume all coordinates are rational.  The value $\Rtwo(S)$ can be evaluated exactly by constructing a planar subdivision of $\DeltaTwo$ on which the active point and active objective coordinate of the lower Tchebycheff envelope are fixed.  The procedure is polynomial in $|S|$ for fixed dimension three, although with a high worst-case degree $O(n^6)$.
\end{theorem}

\begin{proof}[Proof sketch]
Each function $g_p(w)=\max_i w_ip_i$ is convex and piecewise linear over $\DeltaTwo$.  Its internal break lines are line segments on which $w_ip_i=w_jp_j$.  The lower envelope $\tau_S(w)=\min_{p\in S}g_p(w)$ changes either at an internal break line of some $g_p$, or at a point where two candidate functions have equal value.  The equality $g_p(w)=g_q(w)$ is a constant-complexity piecewise-linear curve, because each side is the maximum of three linear forms.  Inserting all such break and equality curves into a planar arrangement yields cells on which $\tau_S(w)$ is a fixed linear function $w_ip_i$.  The integral over each polygonal cell is the integral of a linear function over a rational polygon and can be evaluated exactly by triangulation.  Summing over cells gives $\Rtwo(S)$. (See appendix for a detailed analysis and numerical test and implementation.)
\end{proof}


\section{Algorithmic directions}\label{sec:algorithms}

The two main results have complementary implications.  The NP-hardness theorem rules out a general polynomial-time exact algorithm unless $\mathrm{P}=\mathrm{NP}$.  The submodularity theorem, together with the ideal-point gap theorem, Theorem~\ref{thm:ideal-r2-greedy-gap}, gives a robust approximation framework for the original fixed-ideal-point $\Rtwo$ values.

\paragraph{Greedy and lazy greedy.}
The simplest method is greedy selection by largest marginal integral-$\Rtwo$ improvement.  Since marginal gains decrease, lazy priority-queue implementations can reduce the number of exact marginal evaluations.  Appendix~\ref{app:greedy-implementation} documents the accompanying greedy implementation that uses the exact three-objective evaluator as a marginal-gain oracle.

\paragraph{Quadrature and sampling.}
The integral over $\DeltaTwo$ can be approximated by deterministic quadrature or Monte Carlo sampling.  Sampling turns the problem into a finite maximum-coverage-type problem over sampled weights: each candidate contributes a vector of scalar gains, and the objective is the sum over sampled weights of the maximum selected gain.

\paragraph{Exact arrangement evaluation.}
For small and medium-sized instances, exact planar arrangement integration is possible in fixed dimension three.  This can provide exact marginal gains for greedy or exact objective values for branch-and-bound search.

\paragraph{Branch-and-bound using submodular upper bounds.}
For exact search, submodularity gives simple upper bounds: from a partial set $S$, the best additional gain from $r$ more points is at most the sum of the $r$ largest singleton marginal gains with respect to $S$.  This yields a branch-and-bound method, although worst-case exponential behavior should be expected.

\paragraph{Approximation beyond greedy.}
The weighted anchored-box perspective suggests that geometric approximation schemes for hypervolume-style box selection may inspire stronger approximation algorithms for special subclasses of integral-$\Rtwo$ instances.  However, the Tchebycheff-shadow family is not identical to ordinary box-union volume, so such improvements require separate analysis.

\section{Conclusion}

Three-objective integral-$\Rtwo$ subset selection has a clean scalarization-space interpretation and a useful reciprocal-space interpretation.  The scalarization-space view shows that improvements are volumes of unions of Tchebycheff shadows, which immediately implies monotone submodularity and the classical greedy $(1-1/e)$ approximation guarantee.  For the ordinary ideal-point based indicator this can be stated without changing the optimization target: once the ideal point, weight distribution, and dominated anchor value $R_0$ are fixed, greedy closes at least a $(1-1/e)$-fraction of the maximum possible $\Rtwo$ gap between $R_0$ and the optimal cardinality-$k$ value.  Equivalently, the residual gap to the optimal $\Rtwo$ value is at most $1/e$ of the initial anchor-to-optimum gap.  This is a main positive result of the paper and should be understood as a gap guarantee for the original $\Rtwo$ values, not as a multiplicative approximation of the residual $\Rtwo$ value itself.

The reciprocal-space view gives the complementary negative result.  Exact minimization is computationally hard: by a perspective transformation, improvements become weighted anchored-box gains, and the three-dimensional anchored-box independent-set gadget can be adapted to the density $(x_1+x_2+x_3)^{-4}$.  This yields NP-hardness for exact fixed-cardinality integral-$\Rtwo$ subset selection in three objectives.

Together these facts give a concise algorithmic picture: The two-objective case has exact polynomial-time algorithms because the weight domain is one-dimensional and ordered. With three objectives, exact selection is NP-hard, but the fixed-ideal-point improvement objective remains monotone submodular and thus supports a principled polynomial-time greedy approximation. Future work should address exact and approximate or incremental evaluation, practical lazy-greedy and branch-and-bound methods, and special geometric cases with exploitable structure.

\appendix
\section{Reference implementation of arrangement-based exact evaluation}\label{app:implementation}

This appendix documents the small Python implementation accompanying this report.  The implementation is intended as a transparent reference implementation of Theorem~\ref{thm:arrangement}; it is not intended to compete with specialized Quick-$R_2$.  The code is included in the github repository referenced at the end of the paper.

\subsection{Algorithm implemented}

For a finite set $S\subset\RR^3_{>0}$, the code evaluates
\[
\Rtwo(S)=\int_{\DeltaTwo} \tau_S(w)\,\dd w,
\qquad
\tau_S(w)=\min_{p\in S}\max_{i=1,2,3} w_i p_i,
\]
with the simplex represented by coordinates $(w_1,w_2)$ and $w_3=1-w_1-w_2$.  The implementation uses the following refinement of the proof of Theorem~\ref{thm:arrangement}.

\begin{enumerate}[label=\textbf{Step \arabic*.},leftmargin=*]
\item For each point $p\in S$, create the three affine pieces
\[
  w_1p_1,
  \qquad
  w_2p_2,
  \qquad
  (1-w_1-w_2)p_3.
\]
\item Insert all equality lines between all affine pieces into the simplex.  This arrangement refines the lower-envelope subdivision.  On every resulting cell, the ordering of all affine pieces is fixed.
\item For each polygonal cell, choose an interior representative point.  At this representative point, determine the active point $p$ and active coordinate $i$ realizing
\(
\tau_S(w)=w_i p_i.
\)
Because the arrangement fixes the ordering of all affine pieces, this active affine function is valid throughout the cell, except possibly on cell boundaries of measure zero.
\item Integrate this affine function over the polygon.  If $a w_1+b w_2+c$ is active on a polygon $C$, then
\[
  \int_C (a w_1+b w_2+c)\,\dd w
  = \operatorname{area}(C)\,(a\bar w_1+b\bar w_2+c),
\]
where $(\bar w_1,\bar w_2)$ is the centroid of $C$.
\item Sum these cell integrals.
\end{enumerate}

The current code uses floating-point polygon operations through Shapely.  Thus it implements the exact subdivision principle with ordinary numerical geometric predicates.  A fully rational implementation would replace the polygon clipping and centroid computations by rational arithmetic on line intersections.

\paragraph{Time Complexity.}For a set of \(n\) points in three objectives, the implementation forms the
arrangement induced by all pairwise equality lines of the \(3n\) affine
pieces \(w_1p_1,w_2p_2,(1-w_1-w_2)p_3\).  Hence the number of lines is
\(L=O(n^2)\), and the number of arrangement cells is \(O(L^2)=O(n^4)\).
The reference implementation constructs the subdivision by incremental
polygon splitting.  Since after \(i\) inserted lines there are at most
\(O(i^2)\) cells, the total splitting work is bounded by
\(\sum_{i=1}^L O(i^2)=O(L^3)=O(n^6)\).  Active-piece identification and
cell integration require \(O(n)\) work per cell, hence \(O(n^5)\) time,
and are dominated by the arrangement construction.  Thus the implemented
exact evaluator has worst-case time \(O(n^6)\) and space \(O(n^4)\).
For an improvement \(I_B(S)=R_2(B)-R_2(B\cup S)\), after caching
\(R_2(B)\), the cost is \(O((|B|+|S|)^6)\).

\subsection{Monte Carlo validation}

The verification script compares the arrangement value to independent Monte Carlo integration over $\DeltaTwo$.  The reported values use the unnormalized Lebesgue measure on $\DeltaTwo$, whose area is $1/2$.  Therefore a probability-average version of $\Rtwo$ over the simplex is obtained by multiplying the reported integral values by $2$.

\begin{table}[H]
\centering
\begin{tabular}{lrrrrrrr}
\toprule
case & $n$ & exact & Monte Carlo & s.e. & abs. error & cells & lines\\
\midrule
hand4   & 4 & 0.141937186 & 0.141891084 & $6.18\cdot 10^{-5}$ & $4.61\cdot 10^{-5}$ & 751 & 51\\
random3 & 3 & 0.140613745 & 0.140578829 & $8.51\cdot 10^{-5}$ & $3.49\cdot 10^{-5}$ & 244 & 30\\
random4 & 4 & 0.128206768 & 0.128035663 & $7.93\cdot 10^{-5}$ & $1.71\cdot 10^{-4}$ & 753 & 51\\
\bottomrule
\end{tabular}
\caption{Validation of the reference arrangement evaluator against Monte Carlo integration with 200,000 samples.  The absolute errors are of the same scale as the Monte Carlo standard errors, as expected.}
\label{tab:implementation-validation}
\end{table}

The table shows that the direct arrangement implementation agrees with Monte Carlo integration within sampling error on the tested small instances.  The number of cells already reaches several hundred for four points, which illustrates why this code should be viewed as a reference implementation of the geometric principle rather than as a tuned large-scale evaluator.

\subsection{Reproducibility}

The included \texttt{README.md} describes the command-line use.  In brief, from the artifact folder run
\begin{verbatim}
python3 r2_exact3d.py
python3 test_r2_exact3d.py
\end{verbatim}
The second command writes \texttt{experiment\_results.csv}.  The only nonstandard dependencies are \texttt{numpy} and \texttt{shapely}.  The examples use small point sets because the full arrangement of equality lines grows quickly with the number of candidate points.

\section{Reference implementation of greedy subset selection}\label{app:greedy-implementation}

This appendix documents the accompanying Python implementation of the greedy algorithm for fixed-cardinality integral-$\Rtwo$ subset selection in three objectives.  The implementation is included as \texttt{greedy\_r2\_3d.py}, with a validation script \texttt{test\_greedy\_r2\_3d.py}.  It uses the exact arrangement evaluator from Appendix~\ref{app:implementation} as a black-box oracle for evaluating candidate subsets and marginal improvements.

\subsection{Greedy rule implemented}

Let $P\subset\RR^3_{>0}$ be the candidate set and let $k$ be the desired subset size.  The code fixes a dominated anchor $\bar p$ and sets
\[
 R_0=\Rtwo(\{\bar p\}),
 \qquad
 Q_{\bar p}(S)=R_0-\Rtwo(S).
\]
By default, the anchor is chosen coordinatewise above all candidates,
\[
 \bar p_i = 1.05\max_{p\in P}p_i+10^{-9},\qquad i=1,2,3,
\]
which guarantees $g_{\bar p}(w)\geq g_p(w)$ for all candidates $p$ and all simplex weights $w$ in the minimization/loss convention used in this report.

The greedy algorithm starts with $S_0=\emptyset$.  At step $t$, it chooses
\[
 p_t \in \arg\max_{p\in P\setminus S_{t-1}}
 \bigl\{Q_{\bar p}(S_{t-1}\cup\{p\})-Q_{\bar p}(S_{t-1})\bigr\}.
\]
Since $R_0$ is constant, this is equivalent, after the first step, to choosing the largest exact decrease in the original ideal-point based $\Rtwo$ value,
\[
 p_t \in \arg\max_{p\in P\setminus S_{t-1}}
 \bigl\{\Rtwo(S_{t-1})-\Rtwo(S_{t-1}\cup\{p\})\bigr\}.
\]
At the first step, it is equivalent to selecting the singleton with the smallest $\Rtwo$ value, or equivalently the singleton with the largest anchor-normalized gain.

\begin{pseudocode}{Reference greedy subset selection using exact three-objective $\Rtwo$ evaluation}
\label{alg:greedy-reference-implementation}
\textbf{Input:} Candidate set $P\subset\RR^3_{>0}$, cardinality $k$, dominated anchor $\bar p$.\par
\textbf{Output:} Greedy subset $S$ of size $k$.\par\medskip
$R_0\gets \Rtwo(\{\bar p\})$ and $S\gets\emptyset$.\par
\textbf{for} $t=1,\ldots,k$ \textbf{do}\par
\quad \textbf{for each} $p\in P\setminus S$ \textbf{do}\par
\quad\quad \textbf{if} $S=\emptyset$ \textbf{then} $\Delta(p)\gets R_0-\Rtwo(\{p\})$.\par
\quad\quad \textbf{else} $\Delta(p)\gets \Rtwo(S)-\Rtwo(S\cup\{p\})$.\par
\quad \textbf{end for}\par
\quad choose $p_t\in\argmax_{p\in P\setminus S}\Delta(p)$ and set $S\gets S\cup\{p_t\}$.\par
\textbf{end for}\par
\textbf{return} $S$.
\end{pseudocode}

The implementation caches exact evaluations of previously seen subsets.  This is sufficient for the small validation cases in this artifact.  For larger instances, one would normally replace this direct implementation by a lazy-greedy version, an incremental exact evaluator, or Monte Carlo estimates of marginal gains.

\subsection{Validation against brute force}

For small instances, the script also enumerates all cardinality-$k$ subsets and compares the greedy solution to the exact optimum.  This brute-force comparison is not part of the greedy algorithm; it is included only as a validation tool.  The same selected greedy subset is also checked by Monte Carlo integration of its $\Rtwo$ value.

\begin{table}[H]
\centering
\begin{tabular}{lrrrrrrr}
\toprule
case & $n$ & $k$ & greedy set & optimal set & greedy $\Rtwo$ & optimal $\Rtwo$ & gain ratio\\
\midrule
hand4\_k2   & 4 & 2 & $\{2,0\}$   & $\{0,2\}$   & 0.167597707 & 0.167597707 & 1.000\\
hand4\_k3   & 4 & 3 & $\{2,0,3\}$ & $\{0,2,3\}$ & 0.149428491 & 0.149428491 & 1.000\\
random4\_k2 & 4 & 2 & $\{3,0\}$   & $\{0,3\}$   & 0.127915995 & 0.127915995 & 1.000\\
\bottomrule
\end{tabular}
\caption{Small validation cases for the exact greedy implementation.  The gain ratio is $Q_{\bar p}(G_k)/Q_{\bar p}(S_k^\star)$; in all three tests greedy finds an optimal subset.  The theoretical guarantee from Theorem~\ref{thm:ideal-r2-greedy-gap} is the weaker universal lower bound $1-1/e\approx0.632$.}
\label{tab:greedy-implementation-validation}
\end{table}

The output file \texttt{greedy\_experiment\_results.csv} additionally records the anchor value $R_0$, the exact greedy and optimal gains, the number of brute-force combinations, runtimes, and a Monte Carlo estimate for the greedy subset.  The Monte Carlo estimates agree with the exact greedy values within sampling error in the included runs.

\subsection{Reproducibility}

From the artifact folder, run
\begin{verbatim}
python3 greedy_r2_3d.py
python3 test_greedy_r2_3d.py
\end{verbatim}
The first command prints one small worked example.  The second command writes \texttt{greedy\_experiment\_results.csv}.  The same dependencies as in Appendix~\ref{app:implementation} are used, namely \texttt{numpy} and \texttt{shapely}.

\section{Python verification of the perspective mapping}\label{app:perspective-verification}

This appendix records a compact numerical verification of the perspective transformation used in Theorem~\ref{thm:perspective}.  The accompanying script \texttt{verify\_perspective\_mapping.py} is intended as a reproducibility check for the identity between scalarization-space improvement and weighted reciprocal-box gain.

For a finite set $S\subset\RR^3_{>0}$, the script checks the pointwise equivalence
\[
 t\leq \tau_S(w)
 \quad\Longleftrightarrow\quad
 x=w/t\notin U(b(S)),
\]
the Jacobian identity
\[
 \dd w\,\dd t=(x_1+x_2+x_3)^{-4}\dd x,
\]
and the improvement identity
\[
 \Rtwo(B)-\Rtwo(B\cup S)
 =
 \int_{U(b(B\cup S))\setminus U(b(B))}
 \frac{\dd x}{(x_1+x_2+x_3)^4}.
\]
The scalarization-side value is computed with the arrangement evaluator from Appendix~\ref{app:implementation}.  The reciprocal-space value is independently estimated by Monte Carlo sampling in a bounding box containing the anchored-box difference region.

The test instance uses
\[
B=\{(1.20,1.30,1.10),(1.05,1.75,1.45)\}
\]
and
\[
S=\{(0.78,1.55,1.35),(1.42,0.82,1.22),(1.30,1.25,0.74)\}.
\]
These points were chosen only to produce a nontrivial improvement region.

\begin{table}[H]
\centering
\begin{tabular}{lr}
\toprule
quantity & value \\
\midrule
Exact $\Rtwo(B)$ & 0.355322708469 \\
Exact $\Rtwo(B\cup S)$ & 0.262508524352 \\
Exact improvement & 0.092814184117 \\
Scalarization Monte Carlo estimate & 0.092739326536 \\
Scalarization Monte Carlo s.e. & $1.72\cdot 10^{-4}$ \\
Weighted reciprocal-box Monte Carlo estimate & 0.092744655206 \\
Weighted reciprocal-box Monte Carlo s.e. & $2.38\cdot 10^{-4}$ \\
Pointwise failures for $B$ & $0/200000$ \\
Pointwise failures for $B\cup S$ & $0/200000$ \\
Maximum relative Jacobian error & $1.41\cdot 10^{-15}$ \\
Arrangement cells for $B$ & 53 \\
Arrangement cells for $B\cup S$ & 1775 \\
\bottomrule
\end{tabular}
\caption{Verification of the perspective mapping.  Both Monte Carlo estimates agree with the exact scalarization-side improvement within sampling error; the pointwise map produced no membership failures in the reported samples.}
\label{tab:perspective-verification}
\end{table}

To reproduce the table, run
\begin{verbatim}
python3 verify_perspective_mapping.py
\end{verbatim}
from the artifact folder.  The script writes \texttt{perspective\_mapping\_results.csv}.  It uses \texttt{numpy}, \texttt{shapely}, and the exact evaluator \texttt{r2\_exact3d.py}.

\paragraph{Reproducibility repository}\label{app:reproducibility-ai}

The source code and reproducibility material accompanying the implementation appendices are maintained at
\[
\texttt{\url{https://github.com/emmerichmtm/EvaluationAndSubsetSelectionIntegralR2in3D}}.
\]

\paragraph{Declaration on the use of generative AI.}
OpenAI generative AI models were used as assistance for software prototyping, debugging support, README drafting, LaTeX editing, and language improvement.   The mathematical arguments, algorithmic choices, code outputs, validation results, and final text were reviewed, checked, and edited by the author.  Responsibility for the content of the manuscript and the accompanying implementation remains with the author.

\end{document}